\documentclass{amsart}
\usepackage{amsmath}
\usepackage{amsfonts}
\usepackage{graphicx}
\usepackage{rotating}

\newtheorem{theorem}{Theorem}[section]
\newtheorem*{theorem A}{Theorem A}
\newtheorem*{theorem B}{N\"olker's Theorem}
\newtheorem{lemma}{Lemma}[section]

\theoremstyle{remark}

\theoremstyle{remark}

\theoremstyle{definition}

\numberwithin{equation}{section}
\def\({\left ( }
\def\){\right )}
\def\<{\left < }
\def\>{\right >}

\input{tcilatex}

\begin{document}
\title{Affine factorable surfaces in isotropic spaces}
\author{Muhittin Evren Aydin$^{1},$ Ayla Erdur$^{2}$, Mahmut Ergut$^{3}$}
\address{$^{1}$ Department of Mathematics, Faculty of Science, Firat
University, Elazig, 23200, Turkey}
\address{$^{2,3}$ Department of Mathematics, Faculty of Science and Art,
Namik Kemal University, Tekirdag 59100, Turkey}
\email{meaydin@firat.edu.tr, aerdur@nku.edu.tr, mergut@nku.edu.tr, }
\thanks{}

\begin{abstract}
In this paper, we study the problem of finding the affine factorable
surfaces in a $3-$dimensional isotropic space with prescribed Gaussian $%
\left( K\right) $ and mean $\left( H\right) $ curvature. Because the
absolute figure two different types of these surfaces appear by permutation
of coordinates. We firstly classify the affine factorable surfaces of type 1
with $K,H$ constants. Afterwards, we provide the affine factorable surfaces
of type 2 with $K=const.$ and $H=0.$ In addition, in a particular case, the
affine factorable surfaces of type 2 with $H=const.$ were obtained.
\end{abstract}

\maketitle

\section{Introduction}

Let $\mathbb{R}^{3}$ be a 3-dimensional Euclidean space with usual
coordinates $\left( x,y,z\right) $ and $w\left( x,y\right) :\mathbb{R}%
^{2}\rightarrow \mathbb{R}$ be a smooth real-valued function of 2 variables.
Then its graph given by $z=w\left( x,y\right) $ is a smooth surface with an
atlas that consists of only the following patch%
\begin{equation*}
\mathbf{r}:\mathbb{R}^{2}\rightarrow \mathbb{R}^{3},\left( x,y\right)
\mapsto \left( x,y,w\left( x,y\right) \right) .
\end{equation*}

Notice also that every surface in $\mathbb{R}^{3}$ is locally a part of the
graph $z=w\left( x,y\right) $ if its normal is not parallel to the $xy-$%
plane.\ Otherwise, the regularity assures that it is a part of the graph $%
x=w\left( y,z\right) $\ or $y=w\left( x,z\right) .$\ See \cite[p. 119]{Pr}.
These graphs are also called \textit{Monge surfaces} \cite[p. 302]{Gr}.

Because our target is to solve prescribed Gaussian $\left( K\right) $ and
mean $\left( H\right) $ curvature type equations in a 3-dimensional
isotropic space $\mathbb{I}^{3}$, it is naturally reasonable to focus on
graph surfaces. By separation of variables, we study the graphs $z=w\left(
x,y\right) =f_{1}\left( x\right) f_{2}\left( y\right) ,$ so-called \textit{%
factorable} or \textit{homothetical} \textit{surface}, for smooth functions $%
f_{1},f_{2}$ of single variable. Many results on the factorable surfaces in
other 3-dimensional spaces were made so far, see \cite%
{AO,BS,GV,JS,LM,ML,W,YLi}.

This kind of surfaces also appear as invariant surfaces in the 3-dimensional
space $\mathbb{H}^{2}\times \mathbb{R}$ which is one of eight homogeneous
geometries of Thurston. More clearly, a certain type of translation surfaces
in $\mathbb{H}^{2}\times \mathbb{R}$ is the graph of $y=f_{1}\left( x\right)
f_{2}\left( y\right) ,$ see \cite[p. 1547]{Yo}. For further details, we
refer to \cite{ILM,LY,LMu,Lo,Lo1,T,YL,YLK}.

Recently, Zong, Xiao, Liu \cite{ZXL} defined \textit{affine factorable
surfaces} in $\mathbb{R}^{3}$ as the graphs $z=f_{1}\left( x\right)
f_{2}\left( y+ax\right) ,$ $a\in \mathbb{R},$ $a\neq 0.$ They obtained these
surfaces with $K=0$ and $H=const.$ It is clear that this class of surfaces
is more general than the factorable surfaces.

In this paper, the problem of determining the affine factorable surfaces in $%
\mathbb{I}^{3}$ with $K,H$ constants is considered. By permutation of the
coordinates, because the absolute figure of $\mathbb{I}^{3},$ two different
types of these surfaces appear, i.e. the graphs of $z=f_{1}\left( x\right)
f_{2}\left( y+ax\right) $ and $x=f_{1}\left( y+az\right) f_{2}\left(
z\right) ,$ called \textit{affine factorable surface} of \textit{type 1 }and 
\textit{2}, respectively. Point out also that such surfaces reduce to the
factorable surfaces in $\mathbb{I}^{3}$ when $a=0.$

In this manner, our first concern is to obtain the affine factorable
surfaces of type 1 with $K,H$ constants. And then, we present some results
relating to the affine factorable surfaces of type 2 with $K=const.$ and $%
H=0.$ Furthermore, in a particular case, the affine factorable surfaces of
type 2 with $H=const.$ were found.

\section{Preliminaries}

In this section, we provide some fundamental properties of isotropic
geometry from \cite{CDV}-\cite{EDH},\cite{MS}-\cite{OGR},\cite{PGM,PO,S1,St}%
. For basics of Cayley-Klein geometries see also \cite{K,OS,Y}.

Let $\left( x_{0}:x_{1}:x_{2}:x_{3}\right) $ denote the homogenous
coordinates in a real 3-dimensional projective space $P\left( \mathbb{R}%
^{3}\right) .$ A \textit{3-dimensional} \textit{isotropic space} $\mathbb{I}%
^{3}$ is a Cayley-Klein space defined in $P\left( \mathbb{R}^{3}\right) $ in
which the absolute figure consist of an \textit{absolute plane} $\omega $
and two \textit{absolute lines} $l_{1},l_{2}$ in $\omega $. Those are
respectively given by $x_{0}=0$ and $x_{0}=x_{1}\pm ix_{2}=0.$ The
intersection point of these complex-conjugate lines is called \textit{%
absolute point}, $\left( 0:0:0:1\right) .$

The group of motions of $\mathbb{I}^{3}$ is given by the $6-$parameter group 
\begin{equation}
\left( x,y,z\right) \longmapsto \left( \tilde{x},\tilde{y},\tilde{z}\right)
:\left\{ 
\begin{array}{l}
\tilde{x}=\theta _{1}+x\cos \theta -y\sin \theta , \\ 
\tilde{y}=\theta _{2}+x\sin \theta +y\cos \theta , \\ 
\tilde{z}=\theta _{3}+\theta _{4}x+\theta _{5}y+z,%
\end{array}%
\right.  \tag{2.1}
\end{equation}%
where $\left( x,y,z\right) $ denote the affine coordinates and $\theta
,\theta _{1},...,\theta _{5}\in \mathbb{R}.$ The \textit{isotropic metric}
induced by the absolute figure is given by $ds^{2}=dx^{2}+dy^{2}.$

Due to the absolute figure there are two types of the lines and the planes:
The \textit{isotropic lines} and\textit{\ planes} which are parallel to $z-$%
axis and\textit{\ }others called \textit{non-isotropic lines and planes.} As
an example the equation $ax+by+cz=d$ determines a non-isotropic (isotropic)
plane if $c\neq 0$ ($c=0$), $a,b,c,d\in \mathbb{R}.$

Let us consider an admissible surface (without isotropic tangent planes).
Then it parameterizes%
\begin{equation*}
\mathbf{r}\left( u,v\right) =\left( x\left( u,v\right) ,y\left( u,v\right)
,z\left( u,v\right) \right) ,
\end{equation*}%
where $x_{u}y_{v}-x_{v}y_{u}\neq 0,$ $x_{u}=\frac{\partial x}{\partial u},$
because the admissibility. Notice that the admissible surfaces are regular
too.

Denote $g$ and $h$ the first and the second fundamental forms, respectively.
Then the components of $g$ are calculated by the induced metric from $%
\mathbb{I}^{3}.$ The unit normal vector is $\left( 0,0,1\right) $ because it
is orthogonal to all non-isotropic vectors. The components of $h$ are given
by%
\begin{equation*}
h_{11}=\frac{\det \left( \mathbf{r}_{uu},\mathbf{r}_{u},\mathbf{r}%
_{v}\right) }{\sqrt{\det g}},\text{ }h_{12}=\frac{\det \left( \mathbf{r}%
_{uv},\mathbf{r}_{u},\mathbf{r}_{v}\right) }{\sqrt{\det g}},\text{ }h_{22}=%
\frac{\det \left( \mathbf{r}_{vv},\mathbf{r}_{u},\mathbf{r}_{v}\right) }{%
\sqrt{\det g}},
\end{equation*}%
where $\mathbf{r}_{uu}=\frac{\partial ^{2}\mathbf{r}}{\partial u\partial u}$%
, etc. Therefore, the \textit{isotropic Gaussian} (or \textit{relative}) and 
\textit{mean curvature} are respectively defined by%
\begin{equation*}
K=\frac{\det h}{\det g},\text{ }H=\frac{%
g_{11}h_{22}-2g_{12}h_{12}+g_{22}h_{11}}{2\det g},
\end{equation*}%
where $g_{ij}$ $\left( i,j=1,2\right) $ denotes the component of $g.$ For
convenience, we call these \textit{Gaussian} and \textit{mean curvatures}.

By a\textit{\ flat} (\textit{minimal}) \textit{surface} we mean a surface
with vanishing Gaussian (mean) curvature.

In the particular case that the surface is the graph $z=w\left( x,y\right) $
parameterized $\mathbf{r}\left( x,y\right) =\left( x,y,w\left( x,y\right)
\right) ,$ the Gaussian and mean curvatures turn to%
\begin{equation}
K=w_{xx}w_{yy}-w_{xy}^{2},\text{ }H=\frac{w_{xx}+w_{yy}}{2}.  \tag{2.2}
\end{equation}%
Accordingly; if one is the graph $x=u\left( y,z\right) $ parameterized $%
\mathbf{r}\left( y,z\right) =\left( w\left( y,z\right) ,y,z\right) ,$ then
these curvatures are formulated by%
\begin{equation}
K=\frac{w_{yy}w_{zz}-w_{yz}^{2}}{w_{z}^{4}},\text{ }H=\frac{%
w_{z}^{2}w_{yy}-2w_{y}w_{z}w_{yz}+\left( 1+w_{y}^{2}\right) w_{zz}}{%
2w_{z}^{3}},  \tag{2.3}
\end{equation}%
where the admissibility assures $w_{z}\neq 0.$

\section{Affine factorable surfaces of type 1}

An \textit{affine factorable surface} \textit{of type 1 }in $\mathbb{I}^{3}$
is a graph surface given by 
\begin{equation*}
z=w\left( x,y\right) =f_{1}\left( x\right) f_{2}\left( y+ax\right) ,a\neq 0.
\end{equation*}%
Let us put $u_{1}=x$ and $u_{2}=y+ax$ in order to avoid confusion while
solving prescribed curvature type equations. By (2.2), we get the Gaussian
curvature as%
\begin{equation}
K=f_{1}f_{2}f_{1}^{\prime \prime }f_{2}^{\prime \prime }-\left(
f_{1}^{\prime }f_{2}^{\prime }\right) ^{2},  \tag{3.1}
\end{equation}%
where $f_{1}^{\prime }=\frac{df_{1}}{du_{1}}$ and $f_{2}^{\prime }=\frac{%
df_{2}}{du_{2}}$ and so on.

Notice that the roles of $f_{1}$ and $f_{2}$ in (3.1) is symmetric and thus
it is sufficient to perform the cases depending on $f_{1}.$

\begin{theorem}
Let an affine factorable surface of type 1 in $\mathbb{I}^{3}$ have constant
Gaussian curvature $K_{0}.$ Then, for $c_{1},c_{2},c_{3}\in \mathbb{R}$, one
of the following happens:

\begin{enumerate}
\item[(i)] $w\left( x,y\right) =c_{1}f_{2}\left( y+ax\right) ;$

\item[(ii)] $w\left( x,y\right) =c_{1}e^{c_{2}x+c_{3}\left( y+ax\right) };$

\item[(iii)] $w\left( x,y\right) =c_{1}x^{\frac{1}{1-c_{2}}}\left(
y+ax\right) ^{\frac{c_{2}}{c_{2}-1}},$ $c_{2}\neq 1;$

\item[(iv)] $\left( K_{0}\neq 0\right) $ $w\left( x,y\right) =\sqrt{%
\left\vert K_{0}\right\vert }x\left( y+ax\right) .$
\end{enumerate}
\end{theorem}

\begin{proof}
We have two cases:

\begin{enumerate}
\item Case $K_{0}=0.$ By (3.1), the item (i) of the theorem is obivous. If $%
f_{1},f_{2}\neq const.,$ (3.1) implies $f_{1}^{\prime \prime }f_{2}^{\prime
\prime }\neq 0.$ Thereby, (3.1) can be rewritten as 
\begin{equation}
\frac{f_{1}f_{1}^{\prime \prime }}{\left( f_{1}^{\prime }\right) ^{2}}%
=c_{1}=-\frac{\left( f_{2}^{\prime }\right) ^{2}}{f_{2}f_{2}^{\prime \prime }%
},  \tag{3.2}
\end{equation}%
where $c_{1}\in \mathbb{R},$ $c_{1}\neq 0.$ If $c_{1}=1$, after solving
(3.2), we obtain 
\begin{equation*}
f_{1}\left( u_{1}\right) =c_{2}\exp \left( c_{3}u_{1}\right) ,\text{ }%
f_{2}\left( u_{2}\right) =c_{4}\exp \left( c_{5}u_{2}\right) ,\text{ }%
c_{2},...,c_{5}\in \mathbb{R}.
\end{equation*}%
This proves the item (ii) of the theorem. Otherwise, i.e. $c_{1}\neq 1,$ by
solving (3.2), we derive%
\begin{equation*}
f_{1}\left( u_{1}\right) =\left[ \left( 1-c_{1}\right) \left(
c_{6}u_{1}+c_{7}\right) \right] ^{\frac{1}{1-c_{1}}},\text{ }f_{2}\left(
u_{2}\right) =\left[ \left( \frac{c_{1}}{c_{1}-1}\right) \left(
c_{8}u_{2}+c_{9}\right) \right] ^{\frac{c_{1}}{c_{1}-1}}
\end{equation*}%
for $c_{6},...,c_{9}\in \mathbb{R}.$ This is the proof of the item (iii) of
the theorem.

\item Case $K_{0}\neq 0.$ (3.1) yields $f_{1},f_{2}\neq const.$ We have two
cases:

\begin{enumerate}
\item Case $f_{1}=c_{1}u_{1}+c_{2},$ $c_{1},c_{2}\in \mathbb{R},$ $c_{1}\neq
0.$ By (3.1), we get $K_{0}=-c_{1}^{2}\left( f_{2}^{\prime }\right) ^{2}$or%
\begin{equation*}
f_{2}\left( u_{2}\right) =\frac{\sqrt{\left\vert -K_{0}\right\vert }}{%
\left\vert c_{1}\right\vert }u_{2}+c_{3},\text{ }c_{3}\in \mathbb{R},
\end{equation*}%
which proves the item (iv) of the theorem.

\item Case $f_{1}^{\prime \prime }\neq 0.$ By symmetry, we have $%
f_{2}^{\prime \prime }\neq 0.$ Dividing (3.1) with $f_{1}f_{1}^{\prime
\prime }\left( f_{2}^{\prime }\right) ^{2}$ follows%
\begin{equation}
K_{0}\left( \frac{1}{f_{1}f_{1}^{\prime \prime }}\right) \left( \frac{1}{%
f_{2}^{\prime }}\right) ^{2}=\frac{f_{2}f_{2}^{\prime \prime }}{\left(
f_{2}^{\prime }\right) ^{2}}-\frac{\left( f_{1}^{\prime }\right) ^{2}}{%
f_{1}f_{1}^{\prime \prime }}.  \tag{3.3}
\end{equation}%
The partial derivative of (3.3) with respect to $u_{1}$ gives%
\begin{equation}
K_{0}\underset{\omega _{1}}{\underbrace{\left( \frac{1}{f_{1}f_{1}^{\prime
\prime }}\right) ^{\prime }}}\left( \frac{1}{f_{2}^{\prime }}\right) ^{2}+%
\underset{\omega _{2}}{\underbrace{\left( \frac{\left( f_{1}^{\prime
}\right) ^{2}}{f_{1}f_{1}^{\prime \prime }}\right) ^{\prime }}}=0.  \tag{3.4}
\end{equation}%
Because $f_{2}^{\prime }\neq const.,$ (3.4) concludes $\omega _{1}=\omega
_{2}=0,$ namely%
\begin{equation*}
f_{1}f_{1}^{\prime \prime }=c_{1},
\end{equation*}%
and%
\begin{equation*}
f_{1}f_{1}^{\prime \prime }=c_{2}\left( f_{1}^{\prime }\right) ^{2},\text{ }%
c_{1},c_{2}\in \mathbb{R},\text{ }c_{1}c_{2}\neq 0.
\end{equation*}%
Comparing last two equations leads to $f_{1}^{\prime \prime }=0,$ which is
not our case.
\end{enumerate}
\end{enumerate}
\end{proof}

From (2.2), the mean curvature follows%
\begin{equation}
2H=\left( 1+a^{2}\right) f_{1}f_{2}^{\prime \prime }+2af_{1}^{\prime
}f_{2}^{\prime }+f_{1}^{\prime \prime }f_{2}.  \tag{3.5}
\end{equation}

\begin{theorem}
Let an affine factorable surface of type 1 in $\mathbb{I}^{3}$ be minimal.
Then, for $c_{1},c_{2},c_{3}\in \mathbb{R}$, either

\begin{enumerate}
\item[(i)] it is a non-isotropic plane; or

\item[(ii)] $w\left( x,y\right) =e^{c_{1}x}\left[ c_{2}\sin \left( \frac{%
c_{1}}{1+a^{2}}\left( y+ax\right) \right) +c_{3}\cos \left( \frac{c_{1}}{%
1+a^{2}}\left( y+ax\right) \right) \right] .$
\end{enumerate}
\end{theorem}

\begin{proof}
If $f_{1}$ is a constant function, then (3.5) immediately yields the item
(i) of the theorem. Suppose that $f_{1}^{\prime }f_{2}^{\prime }\neq 0.$ If $%
f_{1}=c_{1}u_{1}+c_{2},$ $c_{1},c_{2}\in \mathbb{R}$, $c_{1}\neq 0,$ (3.5)
follows the following polynomial equation on $f_{1}$%
\begin{equation*}
2ac_{1}f_{2}^{\prime }+\left[ \left( 1+a^{2}\right) f_{2}^{\prime \prime }%
\right] f_{1}=0,
\end{equation*}%
which concludes $f_{2}^{\prime }=0.$ This is not our case. Then we deduce $%
f_{1}^{\prime \prime }f_{2}^{\prime \prime }\neq 0$ and (3.5) can be divided
by $f_{1}^{\prime }f_{2}^{\prime }$ as follows:%
\begin{equation}
-2a=\left( 1+a^{2}\right) \left( \frac{f_{1}}{f_{1}^{\prime }}\right) \left( 
\frac{f_{2}^{\prime \prime }}{f_{2}^{\prime }}\right) +\left( \frac{%
f_{1}^{\prime \prime }}{f_{1}^{\prime }}\right) \left( \frac{f_{2}}{%
f_{2}^{\prime }}\right) ,  \tag{3.6}
\end{equation}%
The partial derivative of (3.6) with respect to $u_{1}$ gives%
\begin{equation}
\left( 1+a^{2}\right) \left( \frac{f_{1}}{f_{1}^{\prime }}\right) ^{\prime
}f_{2}^{\prime \prime }+\left( \frac{f_{1}^{\prime \prime }}{f_{1}^{\prime }}%
\right) ^{\prime }f_{2}=0.  \tag{3.7}
\end{equation}%
We have two cases:

\begin{enumerate}
\item Case $f_{1}^{\prime }=c_{1}f_{1},$ $c_{1}\in \mathbb{R},$ $c_{1}\neq
0. $ That is a solution for (3.7) and thus (3.6) writes%
\begin{equation}
-2a=\frac{1+a^{2}}{c_{1}}\left( \frac{f_{2}^{\prime \prime }}{f_{2}^{\prime }%
}\right) +c_{1}\left( \frac{f_{2}}{f_{2}^{\prime }}\right) ,  \tag{3.8}
\end{equation}%
which is a homogenous linear second-order ODE with constant coefficients.
The characteristic equation of (3.8) has complex roots $\frac{-c_{1}}{1+a^{2}%
}\left( a\pm i\right) $, so the solution of (3.8) turns 
\begin{equation}
f_{2}\left( u_{2}\right) =c_{2}\cos \left( \frac{c_{1}}{1+a^{2}}u_{2}\right)
+c_{3}\sin \left( \frac{c_{1}}{1+a^{2}}u_{2}\right) .  \tag{3.9}
\end{equation}%
Considering (3.9) with the assumption of Case 1 gives the item (ii) of the
theorem.

\item Case $\left( f_{1}/f_{1}^{\prime }\right) ^{\prime }\neq 0.$ The
symmetry yields $\left( f_{2}/f_{2}^{\prime }\right) ^{\prime }\neq 0.$
Furthermore, $\left( 3.7\right) $ leads to $f_{2}^{\prime \prime
}=c_{1}f_{2},$ $c_{1}\in \mathbb{R}$, $c_{1}\neq 0,$ and substituting it
into (3.6) gives%
\begin{equation}
-2a\frac{f_{2}^{\prime }}{f_{2}}=c_{1}\left( 1+a^{2}\right) \frac{f_{1}}{%
f_{1}^{\prime }}+\frac{f_{1}^{\prime \prime }}{f_{1}^{\prime }}.  \tag{3.10}
\end{equation}%
The fact that the left side of (3.10) is a function of the variable $u_{2}$
gives a contradiction.
\end{enumerate}
\end{proof}

\begin{theorem}
Let an affine factorable surface of type 1 in $\mathbb{I}^{3}$ have nonzero
constant mean curvature $H_{0}.$ Then we have either

\begin{enumerate}
\item[(i)] $w\left( x,y\right) =\frac{H_{0}}{1+a^{2}}\left( y+ax\right) ^{2}$
or

\item[(ii)] $w\left( x,y\right) =\frac{H_{0}}{a}x\left( y+ax\right) .$
\end{enumerate}
\end{theorem}

\begin{proof}
We get two cases:

\begin{enumerate}
\item Case $f_{1}^{\prime }=0.$ Then (3.5) proves the item (i) of the
theorem. If $f_{1}=c_{1}u_{1}+c_{2},$ $c_{1},c_{2}\in \mathbb{R},$ $%
c_{1}\neq 0,$ then by (3.5) we get a polynomial equation on $f_{1}$%
\begin{equation*}
-2H_{0}+2ac_{1}f_{2}^{\prime }+\left[ \left( 1+a^{2}\right) f_{2}^{\prime
\prime }\right] f_{1}=0,
\end{equation*}%
which implies $f_{2}^{\prime \prime }=0$ and thus it turns to%
\begin{equation*}
f_{2}^{\prime }=\frac{H_{0}}{ac_{1}}.
\end{equation*}%
This proves the item (ii) of the theorem.

\item Case $f_{1}^{\prime \prime }\neq 0.$ The symmetry concludes $%
f_{2}^{\prime \prime }\neq 0.$ Then (3.5) can be rearranged as 
\begin{equation}
2H_{0}=\left( 1+a^{2}\right) f_{1}p_{2}\dot{p}_{2}+2ap_{1}p_{2}+f_{2}p_{1}%
\dot{p}_{1},  \tag{3.11}
\end{equation}%
where $p_{i}=\frac{df_{i}}{du_{i}}$ and $\dot{p}_{i}=\frac{dp_{i}}{df_{i}}=%
\frac{f_{i}^{\prime \prime }}{f_{i}^{\prime }}.$ (3.11) can be divided by $%
f_{2}p_{1}$ as 
\begin{equation}
\frac{2H_{0}}{f_{2}p_{1}}=\left( 1+a^{2}\right) \left( \frac{f_{1}}{p_{1}}%
\right) \left( \frac{p_{2}\dot{p}_{2}}{f_{2}}\right) +2a\frac{p_{2}}{f_{2}}+%
\dot{p}_{1}.  \tag{3.12}
\end{equation}%
The partial derivative of (3.12) with respect to $f_{1}$ leads to%
\begin{equation}
\left( \frac{2H_{0}}{f_{2}}\right) \frac{d}{df_{1}}\left( \frac{1}{p_{1}}%
\right) =\left( 1+a^{2}\right) \frac{d}{df_{1}}\left( \frac{f_{1}}{p_{1}}%
\right) \left( \frac{p_{2}\dot{p}_{2}}{f_{2}}\right) +\ddot{p}_{1}. 
\tag{3.13}
\end{equation}%
If $p_{1}=c_{3}f_{1}$, $c_{3}\in \mathbb{R},$ $c_{3}\neq 0,$ then the
right-hand side of (3.13) becomes zero, which is no possible. Thereby,
(3.13) can be rewritten by dividing $\frac{d}{df_{1}}\left( \frac{f_{1}}{%
p_{1}}\right) $ as%
\begin{equation*}
2H_{0}\underset{\omega _{1}}{\underbrace{\left( \frac{1}{f_{2}}\right) }}%
\underset{\omega _{2}}{\underbrace{\frac{\frac{d}{df_{1}}\left( \frac{1}{%
p_{1}}\right) }{\frac{d}{df_{1}}\left( \frac{f_{1}}{p_{1}}\right) }}}=\left(
1+a^{2}\right) \underset{\omega _{3}}{\underbrace{\frac{p_{2}\dot{p}_{2}}{%
f_{2}}}}+\underset{\omega _{4}}{\underbrace{\frac{\ddot{p}_{1}}{\frac{d}{%
df_{1}}\left( \frac{f_{1}}{p_{1}}\right) }}},
\end{equation*}%
which implies $\omega _{i}=const.,$ $i=1,...,4,$ for every pair $\left(
f_{1},f_{2}\right) .$ However this is a contradiction due to $\omega _{1}=%
\dfrac{1}{f_{2}}\neq const.$
\end{enumerate}
\end{proof}

\section{Affine factorable surfaces of type 2}

An \textit{affine factorable surface} \textit{of type 2 }in $\mathbb{I}^{3}$
is a graph surface given by 
\begin{equation*}
z=w\left( x,y\right) =f_{1}\left( y+az\right) f_{2}\left( z\right) ,\text{ }%
a\neq 0.
\end{equation*}%
Put $u_{1}=y+az$ and $u_{2}=z.$ From (2.3), the Gaussian curvature follows 
\begin{equation}
K=\frac{f_{1}f_{2}f_{1}^{\prime \prime }f_{2}^{\prime \prime }-\left(
f_{1}^{\prime }f_{2}^{\prime }\right) ^{2}}{\left( af_{1}^{\prime
}f_{2}+f_{1}f_{2}^{\prime }\right) ^{4}},  \tag{4.1}
\end{equation}%
where $f_{1}^{\prime }=\frac{df_{1}}{du_{1}}$ and $f_{2}^{\prime }=\frac{%
df_{2}}{du_{2}}.$ Notice that the regularity refers to $af_{1}^{\prime
}f_{2}+f_{1}f_{2}^{\prime }\neq 0$.

\begin{theorem}
Let an affine factorable surface of type 2 in $\mathbb{I}^{3}$ have constant
Gaussian curvature $K_{0}$. Then it is flat, i.e. $K_{0}=0,$ and for $%
c_{1},c_{2},c_{3}\in \mathbb{R},$ one of the following occurs:

\begin{enumerate}
\item[(i)] $w\left( y,z\right) =c_{1}f_{1}\left( y+az\right) $, $\frac{%
\partial f_{1}}{\partial z}\neq 0;$

\item[(ii)] $w\left( y,z\right) =c_{1}e^{c_{2}\left( y+az\right) +c_{3}z};$

\item[(iii)] $w\left( y,z\right) =c_{1}\left( y+az\right) ^{\frac{1}{1-c_{2}}%
}z^{\frac{c_{2}}{c_{2}-1}},$ $c_{2}\neq 1.$
\end{enumerate}
\end{theorem}

\begin{proof}
If $K_{0}=0$ in (4.1), then the proofs of the items (i),(ii),(iii) are
similar with these of Theorem 3.1. The contd of the proof is by
contradiction. Suppose that $K_{0}\neq 0$ and hence $f_{1},f_{2}$ must be
non-constants. Afterwards, we use the property that the roles of $%
f_{1},f_{2} $ are symmetric in (4.1). If $f_{1}=c_{1}u_{1}+c_{2},$ $%
c_{1},c_{2}\in \mathbb{R},$ $c_{1}\neq 0,$ then (4.1) turns to a polynomial
equation on $f_{1}$ 
\begin{equation*}
\xi _{1}\left( u_{2}\right) +\xi _{2}\left( u_{2}\right) f_{1}+\xi
_{3}\left( u_{2}\right) f_{1}^{2}+\xi _{4}\left( u_{2}\right) f_{1}^{3}+\xi
_{5}\left( u_{2}\right) f_{1}^{4}=0,
\end{equation*}%
where%
\begin{equation*}
\left. 
\begin{array}{l}
\xi _{1}\left( u_{2}\right) =K_{0}a^{4}c_{1}^{4}f_{2}^{4}+c_{1}^{2}\left(
f_{2}^{\prime }\right) ^{2}, \\ 
\xi _{2}\left( u_{2}\right) =4K_{0}a^{3}c_{1}^{3}f_{2}^{3}f_{2}^{\prime },
\\ 
\xi _{3}\left( u_{2}\right) =6K_{0}a^{2}c_{1}^{2}f_{2}^{2}\left(
f_{2}^{\prime }\right) ^{2}, \\ 
\xi _{3}\left( u_{2}\right) =4K_{0}ac_{1}f_{2}\left( f_{2}^{\prime }\right)
^{3}, \\ 
\xi _{5}\left( u_{2}\right) =K_{0}\left( f_{2}^{\prime }\right) ^{4}.%
\end{array}%
\right.
\end{equation*}%
The fact that each coefficient must vanish contradicts with $f_{2}\neq
const. $ Thereby, we conclude $f_{1}^{\prime \prime }f_{2}^{\prime \prime
}\neq 0.$ Next, put $\omega _{1}=f_{1}f_{1}^{\prime \prime },$ $\omega
_{2}=\left( f_{1}^{\prime }\right) ^{2},$ $\omega _{3}=f_{1}^{\prime },$ $%
\omega _{4}=f_{1}$ in (4.1). After taking partial derivative of (4.1) with
respect to $u_{1},$ one can be rewritten as%
\begin{equation}
\mu _{1}f_{2}^{2}f_{2}^{\prime \prime }+\mu _{2}f_{2}f_{2}^{\prime
}f_{2}^{\prime \prime }-\mu _{3}f_{2}\left( f_{2}^{\prime }\right) ^{2}-\mu
_{4}\left( f_{2}^{\prime }\right) ^{3}=0,  \tag{4.2}
\end{equation}%
where%
\begin{equation}
\left. 
\begin{array}{l}
\mu _{1}=a\left( \omega _{1}^{\prime }\omega _{3}-4\omega _{1}\omega
_{3}^{\prime }\right) , \\ 
\mu _{2}=\omega _{1}^{\prime }\omega _{4}-4\omega _{1}\omega _{4}^{\prime },
\\ 
\mu _{3}=a\left( \omega _{2}^{\prime }\omega _{3}+4\omega _{2}\omega
_{3}^{\prime }\right) , \\ 
\mu _{4}=-\omega _{2}^{\prime }\omega _{4}+4\omega _{2}\omega _{4}^{\prime },%
\text{ }%
\end{array}%
\right.  \tag{4.3}
\end{equation}%
for $\omega _{i}^{\prime }=\frac{d\omega _{i}}{du_{1}},i=1,...,4.$ By
dividing (4.2) with $f_{2}^{2}f_{2}^{\prime },$ we deduce%
\begin{equation}
\frac{f_{2}^{\prime \prime }}{f_{2}^{\prime }}\left( \mu _{1}+\mu _{2}\frac{%
f_{2}^{\prime }}{f_{2}}\right) =\left( \mu _{3}\frac{f_{2}^{\prime }}{f_{2}}%
+\mu _{4}\left( \frac{f_{2}^{\prime }}{f_{2}}\right) ^{2}\right) .  \tag{4.4}
\end{equation}%
To solve (4.4) we have to distinguish several cases:

\begin{enumerate}
\item Case $\frac{f_{2}^{\prime }}{f_{2}}=c_{1}\neq 0,$ $c_{1}\in \mathbb{R}%
. $ Substituting it into (4.1) leads to the polynomial equation on $f_{2}$%
\begin{equation}
c_{1}^{2}\left( f_{1}f_{1}^{\prime \prime }-\left( f_{1}^{\prime }\right)
^{2}\right) -K_{0}\left( af_{1}^{\prime }+c_{1}f_{1}\right) ^{4}f_{2}^{2}=0,
\tag{4.5}
\end{equation}%
where the coefficient $af_{1}^{\prime }+c_{1}f_{1}$ must vanish because $%
K_{0}\neq 0$. This however contradicts with the regularity.

\item Case $\mu _{i}=0,$ $i=1,...,4.$ Because $\mu _{3}=0,\ $we conclude $%
6\left( f_{1}^{\prime }\right) ^{2}f_{1}^{\prime \prime }=0$, which is not
our case.

\item Case $\mu _{1}+\mu _{2}\dfrac{f_{2}^{\prime }}{f_{2}}\neq 0.$ (4.4)
follows%
\begin{equation}
\frac{f_{2}^{\prime \prime }}{f_{2}^{\prime }}=\frac{\mu _{3}\frac{%
f_{2}^{\prime }}{f_{2}}+\mu _{4}\left( \frac{f_{2}^{\prime }}{f_{2}}\right)
^{2}}{\mu _{1}+\mu _{2}\frac{f_{2}^{\prime }}{f_{2}}}.  \tag{4.6}
\end{equation}%
The partial derivative of (4.6) with respect to $u_{1}$ gives a polynomial
equation on $\frac{f_{2}^{\prime }}{f_{2}}$ and the fact that each
coefficient must vanish yields the following system:%
\begin{equation}
\left\{ 
\begin{array}{l}
\mu _{2}^{\prime }\mu _{4}-\mu _{2}\mu _{4}^{\prime }=0, \\ 
\mu _{2}^{\prime }\mu _{3}-\mu _{2}\mu _{3}^{\prime }+\mu _{1}^{\prime }\mu
_{4}-\mu _{1}\mu _{4}^{\prime }=0, \\ 
\mu _{1}^{\prime }\mu _{3}-\mu _{1}\mu _{3}^{\prime }=0.%
\end{array}%
\right.  \tag{4.7}
\end{equation}%
By (4.7), we deduce that $\mu _{3}=c_{1}\mu _{1},$ $\mu _{4}=c_{2}\mu _{2},$ 
$c_{1},c_{2}\in \mathbb{R},$ and 
\begin{equation}
\left( c_{2}-c_{1}\right) \left( \mu _{1}^{\prime }\mu _{2}-\mu _{1}\mu
_{2}^{\prime }\right) =0.  \tag{4.8}
\end{equation}%
We have to consider two cases:

\begin{enumerate}
\item $c_{1}=c_{2}.$ Put $c_{1}=c_{2}=c$ and thus $c$ must be nonzero due to
the assumption of Case 3. Then (4.4) leads to%
\begin{equation*}
\frac{f_{2}^{\prime \prime }}{f_{2}^{\prime }}=c\frac{f_{2}^{\prime }}{f_{2}}%
,
\end{equation*}%
which implies $f_{2}^{\prime }=c_{3}f_{2}^{c},$ $c_{3}\in \mathbb{R}$, $%
c_{3}\neq 0.$ Note that $c\neq 1$ due to Case 1. Hence, (4.1) turns to%
\begin{equation}
\frac{K_{0}}{c_{3}^{2}\left( cf_{1}f_{1}^{\prime \prime }-\left(
f_{1}^{\prime }\right) ^{2}\right) }=\frac{f_{2}^{2c}}{\left( af_{1}^{\prime
}f_{2}+c_{3}f_{1}f_{2}^{c}\right) ^{4}}.  \tag{4.9}
\end{equation}%
The partial derivative of (4.9) with respect to $f_{2}$ concludes 
\begin{equation*}
a\left( c-2\right) f_{1}^{\prime }-cc_{3}f_{1}f_{2}^{c-1}=0,
\end{equation*}%
which yields $c=2$ and $2c_{3}f_{1}f_{2}=0.$ This however is not possible.

\item $c_{1}\neq c_{2}.$ It follows from (4.8) that $\mu _{1}=c_{4}\mu _{2},$
$c_{4}\in \mathbb{R}.$ On the other hand, plugging $\omega _{1}=\omega
_{3}^{\prime }\omega _{4}$ and $\omega _{3}=\omega _{4}^{\prime }$ into the
equation $\mu _{3}-c_{1}\mu _{1}=0$ yields%
\begin{equation}
\left( 6-c_{1}\right) \omega _{3}^{2}\omega _{3}^{\prime }-c_{1}\omega
_{3}\omega _{3}^{\prime \prime }\omega _{4}+4c_{1}\left( \omega _{3}^{\prime
}\right) ^{2}\omega _{4}=0.  \tag{4.10}
\end{equation}%
Dividing (4.10) with $\omega _{3}\omega _{4}\omega _{3}^{\prime }$ gives%
\begin{equation}
\left( 6-c_{1}\right) \frac{\omega _{3}}{\omega _{4}}-c_{1}\frac{\omega
_{3}^{\prime \prime }}{\omega _{3}^{\prime }}+4c_{1}\frac{\omega
_{3}^{\prime }}{\omega _{3}}=0.  \tag{4.11}
\end{equation}%
Taking an integration of (4.11) leads to%
\begin{equation}
\omega _{3}^{\prime }=c_{5}\omega _{3}^{4}\omega _{4}^{\frac{6-c_{1}}{c_{1}}%
},\text{ }c_{5}\in \mathbb{R},\text{ }c_{5}\neq 0,  \tag{4.12}
\end{equation}%
or%
\begin{equation}
\omega _{1}=c_{5}\omega _{3}^{4}\omega _{4}^{\frac{6}{c_{1}}}.  \tag{4.13}
\end{equation}%
Moreover, $\mu _{1}-c_{4}\mu _{2}=0$ implies%
\begin{equation*}
\frac{\omega _{1}^{\prime }}{\omega _{1}}-4\frac{a\omega _{3}^{\prime
}-c_{4}\omega _{4}^{\prime }}{a\omega _{3}-c_{4}\omega _{4}}=0
\end{equation*}%
or%
\begin{equation}
\omega _{1}=c_{6}\left( a\omega _{3}-c_{4}\omega _{4}\right) ^{4},\text{ }%
c_{6}\in \mathbb{R},\text{ }c_{6}\neq 0.  \tag{4.14}
\end{equation}%
Comparing (4.13) and (4.14) leads to%
\begin{equation}
\omega _{3}=\frac{-c_{4}\omega _{4}}{\left( \frac{c_{5}}{c_{6}}\right) ^{%
\frac{1}{4}}\omega _{4}^{\frac{3}{2c_{1}}}-a}.  \tag{4.15}
\end{equation}%
Revisiting (4.12) and taking its integration follows%
\begin{equation}
\omega _{3}^{2}=\frac{1}{c_{7}\omega _{4}^{\frac{6}{c_{1}}}+c_{8}},\text{ }%
c_{7},c_{8}\in \mathbb{R},\text{ }c_{7}\neq 0.  \tag{4.16}
\end{equation}%
After equalizing (4.15) and (4.16), we obtain an equation of the form%
\begin{equation*}
c_{4}^{2}\omega _{4}^{\frac{6+2c_{1}}{c_{1}}}-\left( \frac{c_{5}}{c_{6}}%
\right) ^{\frac{1}{2}}\omega _{4}^{\frac{3}{c_{1}}}-2a\left( \frac{c_{5}}{%
c_{6}}\right) ^{\frac{1}{4}}\omega _{4}^{\frac{3}{2c_{1}}}+c_{8}c_{4}^{2}%
\omega _{4}^{2}-a^{2}=0,
\end{equation*}%
which gives a contradiction because $\omega _{4}=f_{1}$ is an arbitrary
non-constant function.
\end{enumerate}
\end{enumerate}
\end{proof}

By (2.3) the mean curvature is%
\begin{equation}
2H=\frac{\left( f_{1}^{\prime }f_{2}\right) ^{2}f_{1}f_{2}^{\prime \prime
}-2\left( f_{1}^{\prime }f_{2}^{\prime }\right) ^{2}f_{1}f_{2}+\left(
f_{1}f_{2}^{\prime }\right) ^{2}f_{2}f_{1}^{\prime \prime
}+f_{1}f_{2}^{\prime \prime }+2af_{1}^{\prime }f_{2}^{\prime
}+a^{2}f_{1}^{\prime \prime }f_{2}}{\left( af_{1}^{\prime
}f_{2}+f_{1}f_{2}^{\prime }\right) ^{3}}.  \tag{4.17}
\end{equation}

\begin{theorem}
There does not exist a minimal affine factorable surface of type 2 in $%
\mathbb{I}^{3}$, except non-isotropic planes.
\end{theorem}

\begin{proof}
The proof is by contradiction. (4.17) follows%
\begin{equation}
\left( f_{1}^{\prime }f_{2}\right) ^{2}f_{1}f_{2}^{\prime \prime }-2\left(
f_{1}^{\prime }f_{2}^{\prime }\right) ^{2}f_{1}f_{2}+\left(
f_{1}f_{2}^{\prime }\right) ^{2}f_{2}f_{1}^{\prime \prime
}+f_{1}f_{2}^{\prime \prime }+2af_{1}^{\prime }f_{2}^{\prime
}+a^{2}f_{1}^{\prime \prime }f_{2}=0.  \tag{4.18}
\end{equation}%
If $f_{1}$ or $f_{2}$ is a constant, then (4.18) deduces that the surface is
a non-isotropic plane. Assume that $f_{1},f_{2}$ are non-constant. If $%
f_{1}^{\prime \prime }=0,$ then (4.18) gives a polynomial equation on $%
f_{1}: $%
\begin{equation*}
2af_{1}^{\prime }f_{2}^{\prime }+\left[ \left( f_{1}^{\prime }f_{2}\right)
^{2}f_{2}^{\prime \prime }-2\left( f_{1}^{\prime }f_{2}^{\prime }\right)
^{2}f_{2}+f_{2}^{\prime \prime }\right] f_{1}=0,
\end{equation*}%
which is no possible because $af_{1}^{\prime }f_{2}^{\prime }\neq 0.$ The
symmetry implies $f_{2}^{\prime \prime }\neq 0.$ Henceforth we deal with the
case $f_{1}^{\prime \prime }f_{2}^{\prime \prime }\neq 0.$ Dividing (4.18)
with $\left( f_{1}^{\prime }f_{2}^{\prime }\right) ^{2}f_{1}f_{2}$ leads to%
\begin{equation}
\left. 
\begin{array}{l}
\frac{f_{2}f_{2}^{\prime \prime }}{\left( f_{2}^{\prime }\right) ^{2}}+\frac{%
f_{1}f_{1}^{\prime \prime }}{\left( f_{2}^{\prime }\right) ^{2}}+\left( 
\frac{1}{\left( f_{1}^{\prime }\right) ^{2}}\right) \left( \frac{%
f_{2}^{\prime \prime }}{f_{2}\left( f_{2}^{\prime }\right) ^{2}}\right) + \\ 
+2a\left( \frac{1}{f_{1}f_{1}^{\prime }}\right) \left( \frac{1}{%
f_{2}f_{2}^{\prime }}\right) +a^{2}\left( \frac{f_{1}^{\prime \prime }}{%
f_{1}\left( f_{1}^{\prime }\right) ^{2}}\right) \left( \frac{1}{\left(
f_{2}^{\prime }\right) ^{2}}\right) =2.%
\end{array}%
\right.  \tag{4.19}
\end{equation}%
The partial derivative of (4.19) with respect to $u_{1}$ and $u_{2}$ yields%
\begin{equation}
\left. 
\begin{array}{l}
\underset{\omega _{1}}{\underbrace{\left( \frac{1}{\left( f_{1}^{\prime
}\right) ^{2}}\right) ^{\prime }}}\underset{\omega _{2}}{\underbrace{\left( 
\frac{f_{2}^{\prime \prime }}{f_{2}\left( f_{2}^{\prime }\right) ^{2}}%
\right) ^{\prime }}}+2a\underset{\omega _{3}}{\underbrace{\left( \frac{1}{%
f_{1}f_{1}^{\prime }}\right) ^{\prime }}}\underset{\omega _{4}}{\underbrace{%
\left( \frac{1}{f_{2}f_{2}^{\prime }}\right) ^{\prime }}} \\ 
+\underset{\omega _{5}}{\underbrace{a^{2}\left( \frac{f_{1}^{\prime \prime }%
}{f_{1}\left( f_{1}^{\prime }\right) ^{2}}\right) ^{\prime }}}\underset{%
\omega _{6}}{\underbrace{\left( \frac{1}{\left( f_{2}^{\prime }\right) ^{2}}%
\right) ^{\prime }}}=0,%
\end{array}%
\right.  \tag{4.20}
\end{equation}%
where $\omega _{1}\omega _{6}\neq 0$ because $f_{1}^{\prime \prime
}f_{2}^{\prime \prime }\neq 0.$ To solve (4.20), we consider two cases:

\begin{enumerate}
\item Case $\omega _{3}=0.$ It follows $f_{1}f_{1}^{\prime }=c_{1},$ $%
c_{1}\in \mathbb{R},$ $c_{1}\neq 0.$ Then, $\omega _{1}=\frac{2}{c_{1}},$ $%
\omega _{5}=\frac{2c_{1}}{f_{1}^{4}}$ and hence (4.20) reduces to the
following polynomial equation on $f_{1}$ 
\begin{equation*}
f_{1}^{4}\omega _{2}+a^{2}c_{1}^{2}\omega _{6}=0,
\end{equation*}%
which is no possible because $\omega _{6}\neq 0.$

\item Case $\omega _{3}\neq 0.$ After dividing (4.20) with $\omega
_{1}\omega _{6},$ we write%
\begin{equation}
\mu _{1}\left( u_{2}\right) +2a\mu _{2}\left( u_{1}\right) \mu _{3}\left(
u_{2}\right) +a^{2}\mu _{4}\left( u_{2}\right) =0,  \tag{4.21}
\end{equation}%
where%
\begin{equation*}
\mu _{1}=\frac{\omega _{2}}{\omega _{6}},\text{ }\mu _{2}=\frac{\omega _{3}}{%
\omega _{1}},\text{ }\mu _{3}=\frac{\omega _{4}}{\omega _{6}},\text{ }\mu
_{4}=\frac{\omega _{5}}{\omega _{1}}.
\end{equation*}%
Notice also that $\mu _{i},$ $i=1,...,4,$ in (4.21) must be constant for
every pair $\left( u_{1},u_{2}\right) .$ Therefore, being $\mu _{2}$ and $%
\mu _{4}$ are constants lead to respectively%
\begin{equation}
f_{1}=\frac{f_{1}^{\prime }}{c_{1}+c_{2}\left( f_{1}^{\prime }\right) ^{2}} 
\tag{4.22}
\end{equation}%
and%
\begin{equation}
\frac{f_{1}^{\prime \prime }}{f_{1}^{\prime }}=f_{1}\left( \frac{c_{3}}{%
f_{1}^{\prime }}+c_{4}f_{1}^{\prime }\right) ,  \tag{4.23}
\end{equation}%
where $c_{1},...,c_{4}\in \mathbb{R}$, $c_{1}\neq 0$ because $\omega
_{3}\neq 0.$ Put $p_{1}=\frac{df_{1}}{du_{1}}$ and $\dot{p}_{1}=\frac{dp_{1}%
}{df_{1}}=\frac{f_{1}^{\prime \prime }}{f_{1}^{\prime }}$ in (4.22) and
(4.23). The derivative of (4.22) with respect to $f_{1}$ gives%
\begin{equation}
\dot{p}_{1}=\frac{\left( c_{1}+c_{2}p_{1}^{2}\right) ^{2}}{c_{1}-c_{2}\left(
p_{1}\right) ^{2}}.  \tag{4.24}
\end{equation}%
Nevertheless, plugging (4.22) into (4.23) leads to%
\begin{equation}
\dot{p}_{1}=\frac{c_{3}+c_{4}p_{1}^{2}}{c_{1}+c_{2}p_{1}^{2}}.  \tag{4.25}
\end{equation}%
Equalizing (4.24) and (4.25) refers to the polynomial equation on $p_{1}$%
\begin{equation*}
\xi _{1}+\xi _{2}p_{1}^{2}+\xi _{3}p_{1}^{4}+\xi _{4}p_{1}^{6}=0,
\end{equation*}%
in which the following coefficients%
\begin{equation*}
\left. 
\begin{array}{l}
\xi _{1}=c_{1}^{3}-c_{1}c_{3}, \\ 
\xi _{2}=3c_{1}^{2}c_{2}-c_{1}c_{4}+c_{2}c_{3}, \\ 
\xi _{3}=3c_{1}c_{2}^{2}+c_{2}c_{4}, \\ 
\xi _{4}=c_{2}^{3}%
\end{array}%
\right.
\end{equation*}%
must vanish. Being $\xi _{2}=\xi _{4}=0$ implies $c_{2}=c_{4}=0$ and thus
from (4.22) and (4.24) we get $f_{1}^{\prime \prime }=c_{1}f_{1}^{\prime
}=c_{1}^{2}f_{1}.$ Considering these ones into (4.17) leads to the
polynomial equation on $f_{1}$%
\begin{equation}
c_{1}^{2}\left[ f_{2}^{2}f_{2}^{\prime \prime }-f_{2}\left( f_{2}^{\prime
}\right) ^{2}\right] f_{1}^{3}+\left[ f_{2}^{\prime \prime
}+2ac_{1}f_{2}^{\prime }+a^{2}c_{1}^{2}f_{2}\right] f_{1}=0,  \tag{4.26}
\end{equation}%
in which the fact that coefficients must vanish yields $f_{2}^{\prime
}+ac_{1}f_{2}=0.$ This however contradicts with the regularity.
\end{enumerate}
\end{proof}

\begin{lemma}
Let an affine factorable surface of type 2 in $\mathbb{I}^{3}$ have nonzero
constant mean curvature $H_{0}.$ If $f_{1}$ or $f_{2}$ is a linear function,
then we have 
\begin{equation}
w\left( y,z\right) =\frac{c_{1}}{\sqrt{\left\vert H_{0}\right\vert }}\sqrt{%
y+az}\text{, }c_{1}\in \mathbb{R}.  \tag{4.27}
\end{equation}
\end{lemma}

\begin{proof}
If $f_{1}=c_{1},$ $c_{1}\in \mathbb{R},$ then (4.17) follows%
\begin{equation}
2H_{0}c_{1}^{2}=\frac{f_{2}^{\prime \prime }}{\left( f_{2}^{\prime }\right)
^{3}}.  \tag{4.28}
\end{equation}%
Solving (4.28) concludes%
\begin{equation*}
f_{2}\left( u_{2}\right) =\frac{-1}{2H_{0}c_{1}^{2}}\sqrt{%
-4H_{0}c_{1}^{2}u_{2}+c_{2}}+c_{3},
\end{equation*}%
for $c_{2},c_{3}\in \mathbb{R}.$ This proves the hypothesis of the lemma. If 
$f_{1}^{\prime }=c_{4}\neq 0,$ $c_{4}\in \mathbb{R},$ then (4.17) reduces to 
\begin{equation*}
2H_{0}\left( ac_{4}f_{2}+f_{1}f_{2}^{\prime }\right)
^{3}=2ac_{4}f_{2}^{\prime }+\left[ c_{4}^{2}f_{2}^{2}f_{2}^{\prime \prime
}-2c_{4}^{2}\left( f_{2}^{\prime }\right) ^{2}f_{2}+f_{2}^{\prime \prime }%
\right] f_{1},
\end{equation*}%
which is a polynomial equation on $f_{1}.$ It is easy to see that the the
coefficient of the term of degree 3 is $\left( f_{2}^{\prime }\right) ^{3}$
which cannot vanish. This completes the proves.
\end{proof}

\section{Conclusions}

The results of the present paper and \cite{AE, Ay} relating to the (affine)
factorable surfaces in $\mathbb{I}^{3}$ with $K,H$ constants are summed up
in Table 1 which categorizes those surfaces. Notice also that, without
emposing conditions, finding the affine factorable surfaces of type 2 with $%
H=const.\neq 0$ is still an open problem.

\begin{sidewaystable}\centering%
$%
\begin{tabular}{lllll}
\hline
\textbf{Properties} & \textbf{FS of type 1} & \textbf{FS of type 2} & 
\textbf{AFS of type 1} & \textbf{AFS of type 2} \\ \hline
${\small K=0}$ & $\left. 
\begin{array}{l}
{\small z=c}_{1}{\small f}_{2}\left( y\right) {\small ;} \\ 
{\small z=c}_{1}{\small e}^{{\small c}_{2}{\small x+c}_{3}{\small y}}{\small %
;} \\ 
{\small z=c}_{1}{\small x}^{\frac{{\small 1}}{{\small 1-c}_{2}}}{\small y}^{%
\frac{{\small c}_{2}}{{\small c}_{2}{\small -1}}}{\small .}%
\end{array}%
\right. $ & $\left. 
\begin{array}{l}
{\small x=c}_{1}{\small f}_{1}\left( {\small z}\right) {\small ;} \\ 
{\small x=c}_{1}{\small e}^{c_{2}y+c_{3}z}{\small ;} \\ 
{\small x=c}_{1}{\small y}^{\frac{{\small 1}}{{\small 1-c}_{2}}}{\small z}^{%
\frac{{\small c}_{2}}{{\small c}_{2}{\small -1}}}{\small .}%
\end{array}%
\right. $ & $\left. 
\begin{array}{l}
{\small z=c}_{1}{\small f}_{2}\left( {\small u}_{2}\right) {\small ;} \\ 
{\small z=c}_{1}{\small e}^{c_{2}x+c_{3}\left( {\small u}_{2}\right) }%
{\small ;} \\ 
{\small z=c}_{1}{\small x}^{\frac{{\small 1}}{{\small 1-c}_{2}}}{\small u}%
_{2}^{\frac{{\small c}_{2}}{{\small c}_{2}{\small -1}}}{\small ,} \\ 
{\small u}_{2}{\small =y+ax.}%
\end{array}%
\right. $ & $\left. 
\begin{array}{l}
{\small x=c}_{1}{\small f}_{1}\left( {\small u}_{1}\right) {\small ;} \\ 
{\small x=c}_{1}{\small e}^{c_{2}{\small u}_{1}+c_{3}z}{\small ;} \\ 
{\small x=c}_{1}{\small u}_{1}^{\frac{{\small 1}}{{\small 1-c}_{2}}}{\small z%
}^{\frac{{\small c}_{2}}{{\small c}_{2}{\small -1}}}{\small ,} \\ 
{\small u}_{1}{\small =y+az.}%
\end{array}%
\right. $ \\ \hline
${\small H=0}$ & $\left. 
\begin{array}{l}
\text{{\small Non-isotropic planes;}} \\ 
{\small z=c}_{1}{\small xy;} \\ 
{\small z=}\left( c_{1}e^{c_{2}x}+c_{3}e^{-c_{2}x}\right) \\ 
{\small \times }\left( c_{4}\cos \left( \text{{\small $c_{2}$$y$}}\right)
+c_{5}\sin \left( \text{{\small $c_{2}$$y$}}\right) \right) {\small .}%
\end{array}%
\right. $ & $\left. 
\begin{array}{l}
\text{{\small Non-isotropic planes;}} \\ 
{\small x=y}\tan \left( c_{1}z\right) {\small ;} \\ 
{\small x=}\frac{c_{1}y}{z}{\small .}%
\end{array}%
\right. $ & $\left. 
\begin{array}{l}
\text{{\small Non-isotropic planes;}} \\ 
{\small z=}e^{c_{1}x}\left[ c_{2}\sin \left( c_{3}{\small u}_{2}\right)
\right. \\ 
\left. +c_{4}\cos \left( c_{5}{\small u}_{2}\right) \right] , \\ 
{\small u}_{2}{\small =y+ax.}%
\end{array}%
\right. $ & {\small Non-isotropic planes.} \\ \hline
${\small K=K}_{0}{\small \neq 0}$ & ${\small z=}\sqrt{\left\vert
K_{0}\right\vert }{\small xy.}$ & $\left. 
\begin{array}{l}
{\small x=}\frac{\pm z}{\sqrt{\left\vert K_{0}\right\vert }y}{\small ;} \\ 
{\small x=}\frac{c_{1}f_{2}\left( z\right) }{y}{\small ,}\text{ } \\ 
{\small z=}\int \sqrt{{\small c}_{2}{\small f}_{2}^{-1}{\small -}\frac{K_{0}%
}{c_{1}^{2}}}{\small df}_{2}\text{{\small .}}%
\end{array}%
\right. $ & ${\small x=}\sqrt{\left\vert {\small K}_{0}\right\vert }{\small x%
}\left( {\small y+ax}\right) {\small .}$ & {\small Non-existence.} \\ \hline
${\small H=H}_{0}{\small \neq 0}$ & ${\small z}=\frac{{\small H}_{0}}{%
{\small c}_{1}}{\small y}^{2},{\small \ }$ & ${\small x=}\pm \sqrt{\frac{-z}{%
H_{0}}}$ & $\left. 
\begin{array}{l}
{\small z=}\frac{{\small H}_{0}}{{\small 1+a}^{2}}\left( {\small y+ax}%
\right) ^{2}{\small ;} \\ 
{\small z=}\frac{{\small H}_{0}}{{\small a}}{\small x}\left( {\small y+ax}%
\right) {\small .}%
\end{array}%
\right. $ & $\left. 
\begin{array}{l}
\text{{\small In the particular case}} \\ 
f_{1}\text{ or }f_{2}\text{ is linear,} \\ 
{\small x=\frac{c_{1}}{\sqrt{\left\vert H_{0}\right\vert }}\sqrt{y+az}.} \\ 
\text{{\small In general case, not }} \\ 
\text{{\small yet known.}}%
\end{array}%
\right. $ \\ \hline
\end{tabular}%
$%
\caption{Factorable surfaces (FS) and affine factorable surfaces (AFS) in
$\mathbb{I}^3$ with $K,H$ constants.}%
\end{sidewaystable}%


\begin{thebibliography}{99}
\bibitem{AO} M. E. Aydin, A.O. Ogrenmis, \textit{Homothetical and
translation hypersurfaces with constant curvature in the isotropic space},
In: Proceedings of the Balkan Society of Geometers \textbf{23} (2015), 1-10.

\bibitem{AE} M.E. Aydin, M. Ergut, \textit{Isotropic geometry of graph
surfaces associated with product production functions in economics}, Tamkang
J. Math. \textbf{47(4)} (2016), 433-443.

\bibitem{Ay} M.E. Aydin,\textit{\ Constant curvature factorable surfaces in
3-dimensional isotropic space,} J. Korean Math. Soc. \textbf{55(1)} (2018),
59-71.

\bibitem{BS} M. Bekkar, B. Senoussi, \textit{Factorable surfaces in the
three-dimensional Euclidean and Lorentzian spaces satisfying} $%
\bigtriangleup r_{i}=\lambda _{i}r_{i},$ J. Geom. \textbf{103} (2012),
17--29.

\bibitem{CDV} B. Y. Chen, S. Decu, L. Verstraelen, \textit{Notes on
isotropic geometry of production models}, Kragujevac J. Math. \textbf{8(1)}
(2014), 23--33.

\bibitem{Da1} L.C.B. Da Silva,\textit{\ Rotation minimizing frames and
spherical curves in simply isotropic and semi-isotropic 3-spaces},
arXiv:1707.06321v2.

\bibitem{Da2} L.C.B. Da Silva,\textit{\ The geometry of Gauss map and shape
operator in simply isotropic and pseudo-isotropic spaces, }
arXiv:1801.01187v1.

\bibitem{DV} S. Decu, L. Verstraelen, \textit{A note on the isotropical
geometry of production surfaces}, Kragujevac J. Math. \textbf{37(2)}\
(2013), 217--220.

\bibitem{D} B. Divjak, \textit{The n-dimensional simply isotropic space},
Zbornik radova \textbf{21} (1996), 33-40.

\bibitem{EDH} Z. Erjavec, B. Divjak, D. Horvat, \textit{The general
solutions of Frenet's system in the equiform geometry of the Galilean,
pseudo-Galilean, simple isotropic and double isotropic space}, Int. Math.
Forum. \textbf{6(1)} (2011), 837-856.

\bibitem{Gr} A. Gray, Modern Differential Geometry of Curves and Surfaces
with Mathematica, CRC Press LLC, 1998.

\bibitem{GV} W. Goemans, I. Van de Woestyne, \textit{Translation and
homothetical lightlike hypersurfaces of semi-Euclidean space}, Kuwait J.
Sci. Eng. \textbf{38 (2A)} (2011), 35-42.

\bibitem{ILM} J. Inoguchi, R. Lopez, M.I. Munteanu,\textit{\ Minimal
translation surfaces in the Heisenberg group Nil$_{3}$}, Geom. Dedicata 
\textbf{161} (2012), 221-231.

\bibitem{JS} L. Jiu, H. Sun, \textit{On minimal homothetical hypersurfaces},
Colloq. Math. \textbf{109} (2007), 239--249.

\bibitem{K} D. Klawitter, Clifford Algebras: Geometric Modelling and Chain
Geometries with Application in Kinematics, Springer Spektrum, 2015.

\bibitem{LY} H. Liu, Y. Yu, \textit{Affine translation surfaces in Euclidean
3-space}, In: Proceedings of the Japan Academy, Ser. A, Mathematical
Sciences, vol. \textbf{89}, pp. 111--113, Ser. A (2013).

\bibitem{LMu} R. Lopez, M.I. Munteanu, \textit{Minimal translation surfaces
in Sol}$_{3}$, J. Math. Soc. Japan \textbf{64(3)} (2012), 985-1003.

\bibitem{LM} R. Lopez, M. Moruz, \textit{Translation and homothetical
surfaces in Euclidean space with constant curvature}, J. Korean Math. Soc. 
\textbf{52(3)} (2015), 523-535.

\bibitem{Lo} R. Lopez, \textit{Minimal translation surfaces in hyperbolic
space}, Beitr. Algebra Geom. \textbf{52(1)} (2011), 105-112.

\bibitem{Lo1} R. Lopez,\textit{\ Separation of variables in equation of mean
curvature type}, Proc. R. Soc. Edinb. Sect. A Math. \textbf{146(5)} (2016),
1017--1035.

\bibitem{ML} H. Meng, H. Liu, \textit{Factorable surfaces in Minkowski space}%
, Bull. Korean Math. Soc. \textbf{46(1)} (2009), 155--169.

\bibitem{MS} Z. Milin-Sipus, \textit{Translation surfaces of constant
curvatures in a simply isotropic space}, Period. Math. Hung. \textbf{68}
(2014), 160--175.

\bibitem{MSD1} Z. Milin-Sipus, B. Divjak, \textit{Curves in }$n-$\textit{%
dimensional }$k-$\textit{isotropic space}, Glasnik Matematicki \textbf{33(53)%
} (1998), 267-286.

\bibitem{MSD2} Z. Milin-Sipus, B. Divjak, \textit{Involutes and evolutes in }%
$n-$\textit{dimensional simply isotropic space}, Zbornik radova \textbf{23(1)%
} (1999), 71-79.

\bibitem{OGR} A.O. Ogrenmis, \textit{Rotational surfaces in isotropic spaces
satisfying Weingarten conditions}, Open Physics \textbf{14(9)} (2016),
221-225.

\bibitem{OS} A. Onishchick, R. Sulanke, Projective and Cayley-Klein
Geometries, Springer, 2006.

\bibitem{PGM} H. Pottmann, P. Grohs, N.J. Mitra, \textit{Laguerre minimal
surfaces, isotropic geometry and linear elasticity}, Adv. Comput. Math. 
\textbf{31} (2009), 391--419.

\bibitem{PO} H. Pottmann, K. Opitz, \textit{Curvature analysis and
visualization for functions defined on Euclidean spaces or surface}s,
Comput. Aided Geom. Design \textbf{11} (1994), 655--674.

\bibitem{Pr} A. Pressley, Elementary Differential Geometry, Springer-Verlag,
London, 2012.

\bibitem{S1} H. Sachs, Isotrope Geometrie des Raumes, Vieweg Verlag,
Braunschweig, 1990.

\bibitem{St} K. Strubecker, \textit{Uber die isotropen Gegenstucke der
Minimalflache von Scherk}, J. Reine Angew. Math. \textbf{293} (1977), 22-51.

\bibitem{T} W. Thurston, Three-dimensional geometry and topology, Princenton
Math. Ser. 35, Princenton Univ. Press, Princenton, NJ, (1997).

\bibitem{W} I. Van de Woestyne, \textit{Minimal homothetical hypersurfaces
of a semi-Euclidean space}, Results. Math. \textbf{27} (1995), 333--342.

\bibitem{Y} I. M. Yaglom, A simple non-Euclidean Geometry and Its Physical
Basis, An elementary account of Galilean geometry and the Galilean principle
of relativity, Heidelberg Science Library. Translated from the Russian by
Abe Shenitzer. With the editorial assistance of Basil Gordon.
Springer-Verlag, New York-Heidelberg, 1979.

\bibitem{Yo} D. W. Yoon, \textit{Minimal translation surfaces in} $\mathbb{H}%
^{2}\times \mathbb{R}$, Taiwanese J. Math. \textbf{17(5)} (2013), 1545-1556.

\bibitem{YL} D. W. Yoon, J. W. Lee, \textit{Translation invariant surfaces
in the 3-dimensional Heisenberg group}, Bull. Iranian Math. Soc. \textbf{%
40(6)} (2014), 1373-1385.

\bibitem{YLK} D. W. Yoon, C. W. Lee, M.K. Karacan, \textit{Some translation
surfaces in the 3-dimensional Heisenberg group}, Bull. Korean Math. Soc. 
\textbf{50(4)} (2013), 1329--1343.

\bibitem{YLi} Y. Yu and H. Liu, \textit{The factorable minimal surfaces},
In: Proceedings of The Eleventh International Workshop on Diff. Geom. 
\textbf{11} (2007), 33-39.

\bibitem{ZXL} P. Zong, L. Xiao, H.L. Liu, \textit{Affine factorable surfaces
in three-dimensional Euclidean space}, Acta Math. Sinica Chinese Serie 
\textbf{58(2)} (2015), 329-336.
\end{thebibliography}
\end{document}